\documentclass[12pt]{amsart}
\usepackage{amsmath}
\usepackage[bmargin=1.2in]{geometry}
\usepackage{amsthm}
\usepackage{amsfonts}
\usepackage{comment}
\usepackage{amssymb}
\usepackage{eufrak}
\usepackage{comment}
\usepackage{hyperref}
\usepackage{mathtools}

\numberwithin{equation}{section}
\usepackage{breqn}

\newtheorem{Theorem}{Theorem}[section]

\newtheorem{Lemma}[Theorem]{Lemma}

\newtheorem*{Conjecture}{Conjecture}
\theoremstyle{definition}

\theoremstyle{remark}
\newtheorem*{remark}{Remark}
\newcommand{\QQ}{\mathfrak{Q}}
\newcommand{\FF}{\mathfrak{F}}
\newcommand{\aaa}{\mathfrak{a}}

\newcommand{\lcm}{\mathrm{lcm}}
\newcommand{\DD}{\mathfrak{D}}

\usepackage{cite}

\title{Shifted distinct-part partition identities in arithmetic progressions}

\author[{\tiny Alwaise}]{Ethan Alwaise}
\address{Department of Mathematics and Computer Science,
Emory University, Atlanta, GA 30322}
\email{ealwais@emory.edu, dicks@emory.edu, jason.friedman@emory.edu, \newline
lianyan.gu@emory.edu, zharner@emory.edu, madeline.locus@emory.edu, iwagner@emory.edu, jwein22@emory.edu}
\address{Department of Mathematics, Harvard University, Cambridge, MA 02138}
\email{hannahlarson@college.harvard.edu}
\author[Dicks]{Robert Dicks}
\author[Friedman]{Jason Friedman}
\author[Gu]{Lianyan Gu}
\author[Harner]{Zach Harner}
\author[Larson]{Hannah Larson}
\author[Locus]{Madeline Locus}
\author[Wagner]{Ian Wagner}
\author[Weinstock]{Josh Weinstock}

\begin{document}

\maketitle

\begin{abstract}
The partition function $p(n)$, which counts the number of partitions of a positive integer $n$, is widely studied. Here, we study partition functions $p_S(n)$ that count partitions of $n$ into distinct parts satisfying certain congruence conditions. 
A shifted partition identity is an identity of the form $p_{S_1}(n-H) = p_{S_2}(n)$ for all $n$ in some arithmetic progression. 
Several identities of this type have been discovered, including two infinite families found by Alladi. In this paper, we use the theory of modular functions to determine the necessary and sufficient conditions for such an identity to exist. In addition, for two specific cases, we extend Alladi's theorem to other arithmetic progressions.
\end{abstract}

\section{Introduction}

A \textit{partition} of an integer $n$ is a non-increasing sequence of positive integers which sum to $n$. 
For example, the partitions of 4 are
\[4, \qquad 3 + 1,\qquad 2 + 2, \qquad 2 + 1 + 1, \qquad 1 + 1 + 1 + 1.\]
A generating function for the \textit{partition function} $p(n)$, which counts the number of partitions of $n$, is given by
\[\prod_{n=1}^\infty \frac{1}{1-q^n} = \sum_{n=0}^\infty p(n)q^n = 1 + q + 2q^2 + 3q^3 + 5q^4 + 7q^5+\ldots.\]
Using similar generating functions, one can obtain interesting identities such as
\[\sum_{n=0}^\infty p_{\text{distinct}}(n)q^n = \prod_{n=1}^\infty (1+q^n) = \prod_{n=1}^\infty\frac{1-q^{2n}}{1-q^n} = \prod_{n=1}^\infty \frac{1}{1-q^{2n-1}} = \sum_{n=0}^\infty p_{\text{odd}}(n)q^n,\]
which shows that the number of partitions of $n$ into distinct parts is equal to the number of partitions of $n$ into odd parts. 

Alladi discovers several identities similar to the classical one above. For example, let $Q_{5}(n)$ be the number of partitions of $n$ into distinct parts congruent to $\pm 1, \pm 5, \pm 7, \pm 9\pmod {24}$, and let $Q_{11}(n)$ be the number of partitions of $n$ into distinct parts congruent to $\pm 1,\pm 7,\pm 9,\pm 11 \pmod {24}$. Similarly, define $\QQ_7(n)$ to be the number of partitions of $n$ into distinct parts which are congruent to $\pm{1}, \pm{7}, \pm{9}, \pm{11} \pmod{30}$ and $\QQ_{13}(n)$ to be the number of partitions of $n$ into distinct parts which are congruent to $\pm{1}, \pm{9}, \pm{11}, \pm{13} \pmod{30}.$  The following table lists the partitions counted by $Q_5(32)$ and $Q_{11}(33)$.
\medskip
\begin{center}
\begin{tabular}{| c | c | c |}
\hline
\rule{0pt}{2.3ex}{Partitions for $Q_5(32)$} & {Partitions for $Q_{11}(33)$} \\[0.4ex]
\hline\hline
\rule{0pt}{2.3ex}19+7+5+1  & 25+7+1  \\ [0.4ex]
\rule{0pt}{2.3ex}17+9+5+1 & 23+9+1 \\[0.4ex]
\rule{0pt}{2.3ex}15+9+7+1 & 17+15+1 \\[0.4ex]
\rule{0pt}{2.3ex}31+1 &17+9+7 \\[0.4ex]
\rule{0pt}{2.3ex}25+7 & 15+11+7 \\[0.4ex]
\rule{0pt}{2.3ex}23+9 & 13+11+9 \\[0.4ex]
\rule{0pt}{2.3ex}17+15 & 33 \\[0.4ex]
\hline
\end{tabular}
\end{center}
\medskip
In particular, this example shows that $Q_5(32) = Q_{11}(33)$. In general, Alladi proves in \cite{A1} that $Q_5(n-1) = Q_{11}(n)$ and $\mathfrak{Q}_{7}(n-1) = \mathfrak{Q}_{13}(n)$ for any positive odd integer $n$. Given this result, it is natural to ask if there exists an arithmetic progression of even $n$ such that $Q_5(n-1)=Q_{11}(n)$ or $\mathfrak{Q}_{7}(n-1) = \mathfrak{Q}_{13}(n)$. Earlier this year, after studying numerical evidence, Ono made the following conjecture.

\begin{Conjecture}[Ono]
We have $Q_5(n-1) = Q_{11}(n)$ and $\QQ_7(n-1) =\QQ_{13}(n)$ for all $n \equiv 4 \pmod 6$. 
\end{Conjecture}

Here, we confirm this conjecture.

\begin{Theorem} \label{true}
Ono's conjecture is true. Moreover, there are infinitely many $n$ and $m$ congruent to $0,2$ mod $6$ such that $Q_5(n-1) \neq Q_{11}(n)$ and $\mathfrak{Q}_7(m-1) \neq \mathfrak{Q}_{13}(m)$.
\end{Theorem}

It turns out that there are many more identities of this type. In \cite{A96}, Alladi finds that his original identity is just one example of two infinite families of shifted partition identities. Alladi proves his theorems using clever manipulations to arrive at beautiful but exotic $q$-series identities. It is natural to ask for a general theory that would prove all shifted distinct-part partition identities. Using the theory of modular functions, in any particular case, one can reduce the proof of these identities to a finite computation of $q$-series coefficients. Indeed, this is the method we employ to prove Theorem \ref{true}.

However, in order to gain a deeper understanding of these partition identities, we develop a framework that determines the necessary and sufficient conditions for such identities to exist. Fix a positive integer $\delta$. For a set $S$ of positive integers strictly between $0$ and $\frac{\delta}{2}$, we define $p_S(n)$ to be the number of partitions of $n$ into distinct parts congruent to elements of $\pm S$ modulo $\delta$.
Now suppose we are given two such sets $S_1$ and $S_2$. In addition, let $R = \{r_1, \ldots, r_m\}$ be a set of residues modulo $\delta$. Let $S_1', S_2',$ and $R'$ be as defined in \eqref{sprime} and \eqref{defR}, and let $X_{S,R}(\tfrac{a}{c};m)$ be the combination of special partition numbers defined in \eqref{defX}.
We say that $S_1$ and $S_2$ are \textit{suited for each other with respect to $R$} if  we have
\[X_{S_1',R'}\left(\tfrac{a}{c};m\right) = X_{S_2',R'}\left(\tfrac{a}{c};m\right),\]
for all $m$ in the finite set of rational numbers in \eqref{M} as $\frac{a}{c}$ runs over cusps of the congruence subgroup $\Gamma_{S_1,S_2}$ defined in \eqref{defG}.
 In addition, we require that $H(S_1,S_2)$, defined in \eqref{defH}, be an integer.
 
 \begin{remark}
 It turns out that for each $S$, the action of the modular group reveals a finite collection of related twisted partition functions $W_S^{(t)}\!\left(\frac{a}{c};n\right)$, which we define in Section 3.
 The conditions for two sets to be suited for each other is equivalent to
 certain combinations of their associated twisted partition numbers being equal.
 \end{remark}

Our main result is the following.

\begin{Theorem} \label{main}
Let $S_1, S_2$ and $R$ be as above, and let $H=H(S_1, S_2)$ be as defined in \eqref{defH}.
We have $p_{S_1}(n) = p_{S_2}(n-H)$ for all $n \equiv r_i \pmod \delta$ if and only if
 $S_1$ and $S_2$ are suited for each other with respect to $R$.
\end{Theorem}
 

\begin{remark}
In \cite{G}, Garvan and Yesilyurt prove several shifted partition identities for partitions into not necessarily distinct parts satisfying certain congruence conditions. Although we consider only partitions into distinct parts, we believe similar results should hold for these other partition functions.
\end{remark}

This paper is organized as follows. In Section \ref{back} we review the relevant theory of modular functions and transformation properties of generalized eta-functions. In Section \ref{genfun}, we introduce generating functions for the partition numbers $p_{S}(n)$ and find explicit expressions for their Fourier expansions at other cusps. We use these formulas to prove Theorem \ref{main} in Section \ref{pf2}. Finally in Section \ref{pf1}, we illustrate another method for proving shifted partition identities in the proof of Theorem \ref{true}.

\section*{Acknowledgements}
The authors would like to thank Ken Ono for advising this project and Michael Mertens for many helpful conversations and suggestions. The authors also thank Emory University and the NSF for their support.

\section{Modular functions} \label{back}

To prove Theorems \ref{true} and \ref{main}, we study certain modular functions which are generating functions for the partition numbers $p_{S}(n)$. In Section \ref{bas} we review the theory of modular forms and modular functions and introduce an operator that sieves on Fourier coefficients. In Section \ref{trans}, we discuss generalized eta-products, which we use to build our generating functions in Section \ref{genfun}.

\subsection{Basics of modular forms and modular functions} \label{bas}
Recall that the \textit{principal congruence subgroup} of $\mathrm{SL}_2(\mathbb{Z})$ of level $N$ is defined by
\[\Gamma(N) := \left\{\left(\begin{matrix} a & b \\ c & d \end{matrix}\right) \in \mathrm{SL}_2(\mathbb{Z}) : \left(\begin{matrix} a & b \\ c & d \end{matrix}\right) \equiv \left(\begin{matrix} 1 & 0 \\ 0 & 1 \end{matrix}\right) \pmod N\right\}.\]
We are also interested in the congruence subgroups
\[\Gamma_1(N) := \left\{\left(\begin{matrix} a & b \\ c & d \end{matrix}\right) \in \mathrm{SL}_2(\mathbb{Z}) : \left(\begin{matrix} a & b \\ c & d \end{matrix}\right) \equiv \left(\begin{matrix} 1 & * \\ 0 & 1 \end{matrix}\right) \pmod N\right\}.\]
These groups act on the extended upper half plane $\mathbb{H}^*= \mathbb{H} \cup \mathbb{Q} \cup \{\infty\}$ by fractional linear transformations
\[\left(\begin{matrix} a & b \\ c & d \end{matrix}\right) \tau := \frac{a\tau+b}{c\tau+d}.\]
Equivalence classes of points in $\mathbb{Q} \cup \{\infty\}$ under the action of a subgroup $\Gamma \subseteq \mathrm{SL}_2(\mathbb{Z})$ are referred to as the \textit{cusps} of $\Gamma$.
Given any $\gamma = \left(\begin{smallmatrix} a&b\\c&d\end{smallmatrix} \right)\in \mathrm{GL}_2(\mathbb{Q})$ with $\det(\gamma) > 0$, the weight-$k$ \textit{slash operator} acts on functions on the upper half plane $\mathbb{H}$ by
\[f(\tau) \vert_k \gamma := (\det\gamma)^{k/2}(c\tau+d)^{-k}f\left(\frac{a\tau+b}{c\tau+d}\right).\]
A \textit{modular form} of integer weight $k$ for a congruence subgroup $\Gamma \subseteq \mathrm{SL}_2(\mathbb{Z})$ of level $N$ is a holomorphic function $f: \mathbb{H} \rightarrow \mathbb{C}$ which satisfies $f(\tau)|_k\gamma = f(\tau)$ for all $\gamma \in \Gamma$ and is holomorphic at the cusps, i.e.~has a Fourier expansion of the form
\[f(\tau)\vert_k \gamma_0 = \sum_{n \geq 0} a_{\gamma_0}(n) q_N^n,\]
where $q_N:=e^{2\pi i \tau/N},$ at each cusp.
A \textit{modular function} is a meromorphic function on $\mathbb{H}$ which is invariant under the weight-$0$ slash operator and meromorphic at the cusps, meaning it has finitely many negative powers of $q$ in its Fourier expansions at the cusps. We refer to this part of the Fourier expansion as the \textit{principal part}.

Since we are interested in partition identities on certain arithmetic progressions, we consider functions that arise from sieving on an arithmetic progression of Fourier coefficients. Suppose we are given a modular form $f(\tau) = \sum_n a_nq^n$ with integer powers of $q$ in its Fourier expansion at infinity and a residue $r$ mod $T$. We define the \textit{sieving operator} $\mathcal{S}_{T,r}$ by
\begin{equation} \label{defS}
f(\tau)\vert \mathcal{S}_{T,r}:= \sum_{n \equiv r \! \! \! \pmod T} a_nq^n =
\frac{1}{T}\sum_{t=0}^{T-1}\zeta_T^{T-tr}f\left(\tau + \frac{t}{T}\right),
\end{equation}
 where $\zeta_T := e^{2\pi i/T}$. The sieved form is again modular of the same weight, but possibly for a different group. The following lemma makes this precise.
 
 \begin{Lemma} \label{s}
 Let $f(\tau)$ be a modular form of weight $k$ (resp.~modular function) on $\Gamma_1(N)$ for some $N$ and let $r$ be any residue modulo $T$. Then $f(\tau)|\mathcal{S}_{T,r}$ is a modular form of the same weight (resp.~modular function) on $\Gamma_1(\mathrm{lcm}(T^2,N))$.
 \end{Lemma}
 
 \begin{proof}
 Suppose we have some $\left(\begin{smallmatrix} a&b\\c&d\end{smallmatrix}\right) \in \Gamma_1(\mathrm{lcm}(T^2,N))$.
Each summand of $f(\tau)|\mathcal{S}_{T,r}$ satisfies
 \begin{align*}
 \left.f\left(\tau+\frac{t}{T}\right)\right|_k\left( \begin{matrix} a&b\\c&d\end{matrix}\right) &= \! \! \! \! \!\left. \begin{array}{c} \\[2ex] \end{array} f(\tau)\right|_k\left(\begin{matrix} 1& \frac{t}{T} \\ 0 & 1 \end{matrix}\right)
\left( \begin{matrix} a&b\\c&d\end{matrix}\right) 
\left(\begin{matrix} 1& \frac{-t}{T} \\ 0 & 1 \end{matrix}\right)
\left(\begin{matrix} 1& \frac{t}{T} \\ 0 & 1 \end{matrix}\right) \\
&= \! \! \! \! \!\left. \begin{array}{c} \\[2ex] \end{array} f(\tau)\right|_k \left(\begin{matrix} a+\frac{ct}{T} & (d-a)\frac{t}{T} - \frac{t^2c}{T^2} \\ c & d-\frac{ct}{T} \end{matrix}\right)\left(\begin{matrix} 1& \frac{t}{T} \\ 0 & 1 \end{matrix}\right) \\
&= \! \! \! \! \!\left. \begin{array}{c} \\[2ex] \end{array} f(\tau)\right|_k \left(\begin{matrix} 1& \frac{t}{T} \\ 0 & 1 \end{matrix}\right) = f\left(\tau+\frac{t}{T}\right),
 \end{align*}
 where the third equality holds because the fact that $T^2$ and $N$ divide $c$ guarantees that the matrix on the left is in $\Gamma_1(N)$. Therefore, $f(\tau)|\mathcal{S}_{T,r}$ is a linear combination of modular forms of weight $k$ on $\Gamma_1(\text{lcm}(T^2,N))$, and hence this function is a modular form of the same weight on the same group. If $f(\tau)$ is holomorphic (resp.~meromorphic), then each summand above is too, and therefore so is $f(\tau) \vert \mathcal{S}_{T,r}$.
 \end{proof}

\subsection{Generalized eta-products} \label{trans}
The \textit{Dedekind eta-function}, defined by
\[\eta(\tau) := q^{1/24} \prod (1 - q^n),\]
where $q:=e^{2\pi i \tau}$, is a modular form of weight $\frac{1}{2}$ on $\mathrm{SL}_2(\mathbb{Z})$ (with a particular multiplier system).
The modular functions we are interested in arise as quotients of products of \textit{generalized eta-functions} that we introduce now.
Let $$P_1(x) := \begin{cases} \{x\} - \frac{1}{2} &\text{if $x \not\in \mathbb{Z}$} \\ 0 & \text{if $x \in \mathbb{Z}$.}\end{cases}$$ 
and $P_2(x):= \{x\}^2-\{x\}+\frac{1}{6}$ be the first and second Bernoulli polynomials respectively evaluated at the fractional part of their arguments, $\{x\} := x - \lfloor x\rfloor$. Following \cite{Robins}, for fixed $\delta$ we define
\begin{equation} \label{defetas}
\eta_{g, h}^{(s)}(\tau) =\alpha_\delta(g,h)q^{\frac{1}{2}P_2\left(\frac{g}{\delta}\right)\tau}\prod_{\substack{m>0\\m\equiv g \! \! \! \pmod \delta}}(1 - \zeta_\delta^hq^{\frac{m}{\delta}}) \prod_{\substack{m> 0 \\ m \equiv -g \! \! \! \pmod \delta}}(1 - \zeta_\delta^{-h}q^{\frac{m}{\delta}}), 
\end{equation}
where $\zeta_\delta := e^{2\pi i /\delta}$ and
\begin{equation} \label{defalpha}
\alpha_\delta(g,h) := \begin{cases} (1 - \zeta_\delta^{-h})e^{\pi i P_1(\frac{h}{\delta})} &\text{if $g \equiv 0$ and $h \not\equiv 0 \!\! \pmod \delta$} \\ 1 & \text{otherwise.}\end{cases} 
\end{equation}
 \begin{remark}
We note that $\eta_{g, h}^{(s)}(\tau)$ is, up to sign, the \textit{Siegel function} $g_{(g/\delta,h/\delta)}(\tau)$ studied by Kubert and Lang in \cite{KL}.
\end{remark}
In addition, we define
\begin{equation}\label{defetas2}
\eta_{\delta, g}(\tau) := \eta_{g,0}^{(s)}(\delta \tau) = e^{\pi i P_2(\frac{g}{\delta})\delta\tau}\prod_{\substack{\ell > 0 \\ \ell \equiv g \! \! \! \pmod \delta}}(1-q^\ell)
\prod_{\substack{\ell > 0 \\ \ell \equiv -g \! \! \! \pmod \delta}}(1-q^\ell)
.\end{equation}

Note that $\eta_{\delta, 0}(\tau) = \eta(\delta\tau)^2$.
From these definitions, it is easy to see that for any integer $t$,
 \begin{equation} \label{td}
 \eta_{\delta,g}\left(\tau+\frac{t}{\delta}\right) = e^{\pi i P_2\left(\frac{g}{\delta}\right)t}\eta_{g,tg}^{(s)}(\delta \tau),
 \end{equation}
 which helps us to describe sieved eta-products.
In \cite{Robins}, Robins studies the transformation properties of the functions $\eta_{\delta,g}(\tau)$ under the action of $\mathrm{SL}_2(\mathbb{Z})$. That is, Robins provides an expression for $\eta_{\delta, g}(A\tau) = \eta_{g,0}^{(s)}(\delta A\tau)$ in terms of an $\eta_{g',h'}^{(s)}(\tau)$, depending on the matrix $A$. Given $A = \left(\begin{smallmatrix} a&b\\c&d\end{smallmatrix}\right)$, let $D=\text{gcd}(c, \delta)$ and choose integers $b_0$ and $d_0$ such that $Dd=\delta a d_0 - cb_0$. In Section 4 of \cite{Robins}, Robins shows that
\[\delta A \tau = A_0\left(\frac{D\tau - b_0}{\delta/D}\right) \qquad \text{where} \qquad A_0 = \left(\begin{matrix} \frac{\delta a}{D} & ab_0+bD \\ \frac{c}{D} & d_0 a \end{matrix}\right).\]

\begin{remark}
Here, we use slightly different notation than Robins. In particular, the roles of $b$ and $b_0$ (resp.~$d$ and $d_0$) have been switched.
\end{remark}
Using the general transformation law from equation (30) of Chapter VIII in \cite{Robins9}, assuming $g \not\equiv 0 \pmod \delta$, we have
\begin{equation} \label{key}
\eta_{g,tg}^{(s)}(\delta A\tau) = \eta_{g,tg}^{(s)}\left(A_0 \left(\frac{D\tau-b_0}{\delta/D}\right)\right)= e^{\pi i \mu_{\delta, g}^{(t)}}\eta_{g',h'}^{(s)}\left(\frac{D^2\tau-Db_0}{\delta}\right),
\end{equation}
where 
$\left(\begin{smallmatrix} g' \\ h'\end{smallmatrix}\right) = A_0^{tr} \left(\begin{smallmatrix} g \\ tg \end{smallmatrix}\right)$, i.e.
\begin{equation} \label{primed}
g' = \frac{g(\delta a + tc)}{D} \qquad \text{and} \qquad  h' = g(ab_0 + bD + td_0a),
\end{equation}
 and
 \begin{align} \label{defmu}
 \mu_{\delta,g}^{(t)} &:= \mu_{\delta,g}^{(t)}(a,b,c,d) \\
 &:= \frac{\delta a}{c} P_2\left(\frac{g}{\delta}\right) + \frac{d_0aD}{c}P_2\left(\frac{g'}{\delta}\right) - 2\sum_{\nu=0}^{\frac{c}{D}-1}
 P_1\left(\tfrac{D(\delta\nu+g)}{\delta c}\right)
  P_1\left(\tfrac{Dg'+\delta^2a\nu}{c}\right). \notag
 \end{align}

When $t=0$, this formula allows Robins to determine a congruence subgroup fixing generalized eta products of the form
\[f(\tau) = \prod_{\delta \mid N} \eta_{\delta, g}^{r_{\delta,g}}(\tau).\]
Namely, Robins finds (see equations (11) and (12) of \cite{Robins}) that $f(\tau)$ is modular on $\Gamma_1(N)$ if
\begin{equation}\label{C1}
\sum_{\substack {\delta \mid N \\ g}}\delta P_2(\tfrac{g}{\delta})r_{\delta,g}\equiv0\pmod{2}
\end{equation}
and
\begin{equation}\label{C2}
\sum_{\substack {\delta \mid N \\ g}}\frac{N}{6\delta} r_{\delta,g}\equiv0\pmod{2}.
\end{equation}
We use this criteria in the following section to prove the modularity of our generating functions.

\section{Generating functions for $p_S(n)$} \label{genfun}
We now study modular functions which are generating functions for the partition numbers $p_S(n)$. We first define these functions in Section \ref{defgen}.
Then in Section \ref{es}, we provide explicit expressions for their Fourier expansions at other cusps.

\subsection{Definition and modularity} \label{defgen}
Given a set $S$ of integers strictly between $0$ and $\frac{\delta}{2}$, we define
\[F_S(\tau) := \prod_{g \in S} \frac{\eta_{2\delta, 2g}(\tau)}{\eta_{\delta,g}(\tau)}
 = q^{\text{ord}_S} \prod_{g \in S} \prod_{\substack{\ell > 0 \\ \ell \equiv \pm g \! \! \! \pmod \delta}}(1+q^\ell)
=q^{\text{ord}_S}   \sum_{n=0}^\infty p_S(n)q^n,\]
where
\begin{equation*}
\text{ord}_S := \frac{1}{2}\sum_{g \in S} \left(2\delta P_2(\tfrac{2g}{2\delta})-\delta P_2(\tfrac{g}{\delta})\right) = \frac{1}{2}\sum_{g \in S} \delta P_2(\tfrac{g}{\delta})
\end{equation*}
is the order at infinity. We first establish the modularity properties of these functions. For any $r \in \mathbb{Q}$ in lowest terms, let $\text{Den}(r)$ be its denominator.

\begin{Lemma} \label{2.1}
Let $v= \mathrm{Den}(\mathrm{ord}_S)$. Then $F_S(v\tau)$ is a modular function on
\[\Gamma_1\left(\frac{24\delta v}{\mathrm{gcd}(|S|, 12)}\right).\]
\end{Lemma}

\begin{proof}
From the definition of $F_S(\tau)$ and the fact that $\eta_{\delta, g}(v\tau) = \eta_{\delta v, gv}(\tau),$ we see that
\[F_S(v\tau) = \prod_{g \in S} \frac{\eta_{2\delta v, 2gv}(\tau)}{\eta_{\delta v,gv}(\tau)}.\]
Hence, the left-hand side of \eqref{C1} is
\begin{equation*}
\sum_{g\in S}2 \delta v P_2(\tfrac{g}{\delta}) - \sum_{g \in S} \delta v P_2(\tfrac{g}{\delta}) = \sum_{g\in S} \delta v P_2(\tfrac{g}{\delta}) = 2v \cdot \text{ord}_S,
\end{equation*}
which is even by our choice of $v = \text{Den}(\text{ord}_S)$.
Next, for $F_S(v\tau)$ the left-hand side of \eqref{C2} turns into
\begin{equation*}
\sum_{g \in S}\frac{N}{6(2 \delta v)}  -\sum_{g\in S}\frac{N}{6\delta v}
= \frac{-N}{12 \delta v}|S|.
\end{equation*}
The minimal choice of $N$ such that the above is even and $2\delta v \mid N$ is 
\[N = \frac{24\delta v}{\mathrm{gcd}(|S|, 12)}.\qedhere\]
\end{proof}

Finally, suppose we are given a set $S$ of integers between $0$ and $\frac{\delta}{2}$ with $\text{ord}_S \in \mathbb{Z}$ and a set $R$ of residues modulo $\delta$.
Then we define sieved generating functions
\begin{equation} \label{defcalF}
\mathcal{F}_{S,R}(\tau) := \sum_{r \in R} F_{S}(\tau) \vert \mathcal{S}_{\delta,r}.
\end{equation}
Lemmas \ref{2.1} and \ref{s} allow us to deduce the following.

\begin{Lemma} \label{2.3}
For any set $S$ of integers between $0$ and $\frac{\delta}{2}$ with $\mathrm{ord}_S\in \mathbb{Z}$ and any set $R$ of residues modulo $\delta$, the function $\mathcal{F}_{S,R}(\tau)$ is modular for
\[\Gamma_1\left(\lcm\left(\delta^2,\frac{24\delta}{\gcd(|S|,12)}\right)\right).\]
 \end{Lemma}

  \subsection{Expansions at other cusps} \label{es}
 
Here, we use the transformation properties of generalized eta products presented in Section \ref{trans} to write down explicit expressions for the Fourier expansions of our generating functions at arbitrary cusps. In particular, our description of their principal parts leads us to the appropriate definition of  when two sets are suited for each other.

Fix a modulus $\delta$ and a set $S$ of integers strictly between $0$ and $\frac{\delta}{2}$.
For each rational number $\frac{a}{c}$ with $\gcd(a, c) = 1$, fix integers $b$ and $d$ such that $ad-bc = 1$.
As in Section \ref{trans}, let $D := \gcd(c, \delta)$
and choose integers $b_0$ and $d_0$ such that $Dd=\delta a d_0 - cb_0$. In addition, let
\[\epsilon := \epsilon(c,\delta) := \frac{\gcd(c,2\delta)}{\gcd(c,\delta)}.\]
Note that if we replace $\delta$ and $g$ by $2\delta$ and $2g$, we substitute $\epsilon D$ for $D$, then $\epsilon b_0$ for $b_0$ and $\frac{\epsilon}{2}d_0$ for $d_0$. In particular, $g'$ and $h'$ defined in \eqref{primed} scale to $\frac{4}{\epsilon}g'$ and $2\epsilon h'$.

We first find an expression for the Fourier expansion at the cusp $\frac{a}{c}$ of each term
\[F_S^{(t)}(\tau) := F_S\left(\tau + \frac{t}{\delta}\right) = \prod_{g \in S} \frac{\eta_{2\delta,2g}(\tau+\frac{2t}{2\delta})}{\eta_{\delta,g}(\tau+\frac{t}{\delta})} = \prod_{g \in S}\frac{\eta_{2g,2gt}^{(s)}(2\delta\tau)}{\eta_{g,gt}^{(s)}(\delta \tau)}\cdot e^{\pi i P_2(\frac{g}{\delta})t} \]
appearing in the sieved form $\mathcal{F}_{S,R}(\tau)$, where the last equality above follows from \eqref{td}. By \eqref{key} and \eqref{defetas}, the order of $F_S^{(t)}(\tau)$ at $\frac{a}{c}$ is
\begin{equation} \label{defordt}
\mathrm{ord}_S^{(t)}\!\left(\tfrac{a}{c}\right) = \frac{D^2}{2\delta} \sum_{g \in S}\left( \frac{\epsilon^2}{2}P_2\left(\frac{2g(\delta a + tc)}{\epsilon \delta D}\right) - P_2\left(\frac{g(\delta a + tc)}{\delta D}\right)\right).
\end{equation}
Also using \eqref{key}, we find that the first coefficient of $F_S^{(t)}(\tau)$ in the Fourier expansion at the cusp $\frac{a}{c}$ is given by
\begin{align*}
Z_S^{(t)}(\tfrac{a}{c}) &:=e\left(\frac{1}{2}\sum_{g \in S} (tP_2(\tfrac{g}{\delta})+\mu_{2\delta,2g}^{(2t)}-\mu_{\delta,g}^{(t)})- \frac{b_0\text{ord}_S^{(t)}(\frac{a}{c})}{D}
 \right)
\cdot \prod_{g \in S} \frac{\alpha_{2\delta}(\frac{4}{\epsilon}g',2\epsilon h')}{\alpha_\delta(g',h')}
 \end{align*}
 where $e(x) := e^{2\pi ix}$, $\mu_{\delta, g}^{(t)}$ is defined in \eqref{defmu}, $\alpha_\delta(g,h)$ is defined in \eqref{defalpha}, and $g'$ and $h'$ are as in \eqref{primed}.

The following function helps us to describe the other coefficients in the Fourier expansion. Let
\[C_{\delta,g}^{(t)}(\tfrac{a}{c};\ell) := \pm\frac{h'}{\delta} - \frac{Db_0}{\delta^2}\ell = \pm \frac{g(ab_0+bD+td_0a)}{\delta} -\frac{Db_0}{\delta^2}\ell,\]
where the sign is determined by the congruence $\ell \equiv \pm \alpha \left(\text{mod} \ \frac{4\delta}{\epsilon^2}\right)$ for some $0 < \alpha < \frac{\delta}{2}$ understood from context.
We note that for any $\ell$, an elementary calculation shows that
\begin{align} \label{4/e}
C^{(2t)}_{2\delta, 2g}\left(\tfrac{a}{c}; \ell \right) 
= \epsilon C_{\delta, g}^{(t)}\left(\tfrac{a}{c}; \tfrac{\epsilon}{4}\ell \right).
\end{align}

In addition, for fixed $a, c, S,\delta$ and $t$ we define a \textit{special partition of $n$} to be a collection of sets
\[\lambda = (\{\lambda_i^{(g)}\})_{g \in S}\]
of distinct positive integers satisfying
\[\lambda_i^{(g)} \equiv \pm \frac{4g'}{\epsilon^2} + (3-\epsilon)\delta 
\ \  \left(\text{mod} \ \frac{4\delta}{\epsilon^2}\right)\]
and
\[\sum_{g \in S} \sum_{i} \lambda_i^{(g)} = n.\]
We then define \textit{twisted special partition numbers}
\begin{equation}
W_S^{(t)}\!\left(\tfrac{a}{c}; n\right) := \sum_{\lambda} e\left(\sum_{g,i} \left(C^{(t)}_{\delta,g}\left(\tfrac{a}{c};\tfrac{\epsilon^2}{4}\lambda_i^{(g)}\right)+\frac{\epsilon}{2}\right)\right),
\end{equation}
where the sum ranges over all special partitions $\lambda = (\{\lambda_i^{(g)}\})_{g \in S}$ of $n$.

Note that $W_S^{(0)}(\infty;n) = p_S(n)$.

\begin{remark} 
The use of twisted or weighted partitions is not new. In fact, in \cite{A96} Alladi makes use of a similar weighting system for partitions depending on congruence properties of their parts.
\end{remark}

With these definitions, we can provide the following formula for the Fourier expansion of $F_S^{(t)}(\tau)$ at the cusp $\frac{a}{c}$.
\begin{Lemma}
Let $A = \left(\begin{smallmatrix} a & b \\ c & d \end{smallmatrix}\right)$. Then we have
\[F_S^{(t)}(A\tau) = Z_S^{(t)}\left(\tfrac{a}{c}\right) q^{\mathrm{ord}_S^{(t)}\left(\frac{a}{c}\right)}\sum_{n=0}^\infty W_S^{(t)}\left(\tfrac{a}{c};n\right)q^{\frac{\epsilon^2D^2}{4\delta^2}n}.\]
\end{Lemma}
\begin{proof}
When we apply \eqref{key} to the terms in $F_S^{(t)}(\tau)$,
we find
\begin{align*}
F_S^{(t)}(A\tau) &=   \prod_{g \in S} \frac{\eta_{2g, 2gt}^{(s)}(2\delta A\tau)}{\eta_{g,gt}^{(s)}(A\tau)}\cdot e^{\pi i P_2(\frac{g}{\delta})t}  \\
&= \prod_{g \in S}\frac{e^{\pi i \mu_{2\delta, 2g}^{(t)}}{\eta_{\frac{4}{\epsilon}g', 2\epsilon h'}^{(s)}}\left(\frac{\epsilon^2(D^2\tau-Db_0)}{2\delta}\right)}{e^{\pi i \mu_{\delta, g}^{(t)}}{\eta_{g', h'}^{(s)}}\left(\frac{D^2\tau - Db_0}{\delta}\right)} \cdot e^{\pi i P_2(\frac{g}{\delta})t}  \\
&=Z_S^{(t)}(\tfrac{a}{c})q^{\mathrm{ord}_S^{(t)}\left(\tfrac{a}{c}\right)}\prod_{g \in S} \dfrac{ \ \prod\limits_{\ell \equiv \pm \frac{4}{\epsilon} g' \! \! \! \pmod {2\delta}} \! \! \left(1-e\left(C_{2\delta, 2g}^{(2t)}\left(\frac{a}{c}; \ell \right)\right)q^{\frac{D^2 \epsilon^2}{4 \delta^2}\ell}\right)}{\! \! \! \! \prod\limits_{k \equiv \pm g' \! \! \! \pmod \delta} \! \left(1-e\left(C_{\delta, g}^{(t)}\left(\frac{a}{c}; k\right)\right)q^{\frac{D^2}{\delta^2}k}\right)}. \\
\intertext{Using \eqref{4/e} on the terms in the numerator, we can write this as}
&= Z_S^{(t)}(\tfrac{a}{c})q^{\mathrm{ord}_S^{(t)}\left(\tfrac{a}{c}\right)}\prod_{g \in S} \frac{ \ \prod\limits_{\ell \equiv \pm {\frac{4}{\epsilon}g'} \! \! \! \pmod {2\delta}} \! \!  \left(1-e\left(\epsilon C_{\delta, g}^{(t)}\left(\frac{a}{c}; {\frac{\epsilon}{4}\ell} \right)\right)q^{\frac{D^2 \epsilon^2}{4 \delta^2}\ell}\right)}{\prod\limits_{k \equiv \pm g' \! \! \! \pmod \delta}\left(1-e\left(C_{\delta, g}^{(t)}\left(\frac{a}{c}; k\right)\right)q^{\frac{D^2}{\delta^2}k}\right)}.
\end{align*}
To simplify this further we consider the cases $\epsilon=1$ and $\epsilon =2$ separately. When $\epsilon = 2$, we see
\[F_S^{(t)}(A\tau) = Z_S^{(t)}(\tfrac{a}{c})q^{\mathrm{ord}_S^{(t)}\left(\tfrac{a}{c}\right)}\prod_{g \in S} \frac{ \ \prod\limits_{\ell \equiv \pm 2g' \! \! \! \pmod {2\delta}} \!  \left(1-e\left(2C_{\delta, g}^{(t)}\left(\frac{a}{c}; {\frac{\ell}{2}} \right)\right)q^{\frac{D^2}{\delta^2}\ell}\right)}{\prod\limits_{ k \equiv \pm g' \! \! \! \pmod \delta}\left(1-e\left(C_{\delta, g}^{(t)}\left(\frac{a}{c}; k\right)\right)q^{\frac{D^2}{\delta^2}k}\right)}.\]
For each term with $k \equiv \pm g' \pmod \delta$ in the denominator, there is a term with $\ell=2k \equiv {\pm 2g'} \pmod {2\delta} $ in the numerator.  Thus, we have a product over terms
\[\frac{1-e\left(2C_{\delta, g}^{(t)}\left(\frac{a}{c}; k\right)\right)q^{\frac{D^2}{\delta^2}2k}}{1-e\left(C_{\delta, g}^{(t)}\left(\frac{a}{c}; k\right)\right)q^{\frac{D^2}{\delta^2}k}} = 1 + e\left(C_{\delta, g}^{(t)}\left(\tfrac{a}{c}; k\right)\right)q^{\frac{D^2}{\delta^2}k}.\]
Hence, when $\epsilon = 2$ we have
\[F_S^{(t)}(A\tau) = Z_S^{(t)}\left(\tfrac{a}{c}\right)q^{\mathrm{ord}_S^{(t)}\left(\tfrac{a}{c}\right)}\prod_{g \in S} \prod_{\substack{k > 0 \\ k \equiv \pm g' \! \! \! \pmod \delta}}{ \! \! \!  \! \! \left(1 + e\left(C_{\delta, g}^{(t)}\left(\tfrac{a}{c}; k \right)\right)q^{\frac{D^2}{\delta^2} k}\right)}.\]

Similarly, when $\epsilon = 1$, for each term with $k\equiv \pm g' \pmod \delta$ in the denominator, we have a term $\ell=4k \equiv {4g' \pmod {4\delta}}$ in the numerator which cancels it.  However, the terms with $\ell \equiv {4g' + 2\delta} \pmod {4\delta}$ remain.  This shows that for $\epsilon = 1$,
\[F_S^{(t)}(A\tau) = Z_S^{(t)}(\tfrac{a}{c})q^{\mathrm{ord}_S^{(t)}\left(\tfrac{a}{c}\right)}\prod_{g \in S} \! \! \prod_{\substack{\ell > 0 \\ \ell \equiv \pm {4g' + 2\delta} \pmod {4\delta}}}{ \! \! \! \! \! \! \!  \left(1 - e\left(C_{\delta, g}^{(t)}\left(\tfrac{a}{c}; \tfrac{\ell}{4} \right)\right)q^{\frac{D^2}{4 \delta^2} \ell}\right)}.\]

We can now see that in both cases, the twisted special partition numbers describe the coefficients in the products.
\end{proof}

By the above lemma, the coefficient of $q^m$ in $F_S^{(t)}(\tau)$ is given by
\begin{equation}
Y_S^{(t)}\!\left(\tfrac{a}{c};m\right) := Z_S^{(t)}(\tfrac{a}{c})W_S^{(t)}\left(\frac{a}{c};\frac{4\delta^2}{\epsilon^2D^2}(m-\text{ord}_S^{(t)}(\tfrac{a}{c}))\right)
\end{equation}
where $W_S^{(t)}(\frac{a}{c};n) := 0$ if $n \notin \mathbb{Z}$. Now suppose that $\text{ord}_S \in \mathbb{Z}$. The coefficients in the expansion of the sieved form $\mathcal{F}_{S,R}(\tau)$ at an arbitrary cusp are expressible in terms of various $Y_S^{(t)}(\frac{a}{c};m)$. Define
\begin{equation} \label{defX}
X_{S,R}\left(\tfrac{a}{c}; m\right) 
:=\frac{1}{\delta} \sum_{t=0}^{\delta-1} \left(\sum_{r \in R} \zeta_{\delta}^{\delta-tr}\right)
Y_{S}^{(t)}\!\left(\tfrac{a}{c}; m\right).
\end{equation}
Then we have the following.

\begin{Lemma} \label{4.2}
Suppose $\mathrm{ord}_S \in \mathbb{Z}$ and $R$ is a set of residues modulo $\delta$.
The Fourier expansion of $\mathcal{F}_{S,R}(\tau)$ at the cusp $\frac{a}{c}$ is given by
\[\mathcal{F}_{S,R}(A\tau) = \sum_{m \in M} X_{S,R}\left(\tfrac{a}{c};m\right) q^m,\]
where
\[M:= \left\{\mathrm{ord}_{S}^{(t)}\left(\tfrac{a}{c}\right)+\frac{\epsilon^2D^2}{4\delta^2}n :  n \in \mathbb{Z}_{\geq 0}, \  t=0, \ldots, \delta -1 \right\}.\]
\end{Lemma}

 \begin{proof}
Let $A=\left(\begin{smallmatrix} a&b\\c&d\end{smallmatrix}\right) \in \mathrm{SL}_2(\mathbb{Z})$. From the definition of $\mathcal{F}_{S,R}(\tau)$, we have 
\begin{align*}
\mathcal{F}_{S,R}(A\tau) &= \frac{1}{\delta}\sum_{r \in R}\sum_{t=0}^{\delta-1}\zeta_\delta^{\delta-tr} F_S^{(t)}(A\tau)
=\frac{1}{\delta}\sum_{r \in R}\sum_{t=0}^{\delta-1}\zeta_\delta^{\delta-tr}\left(\sum_{m}Y_S^{(t)}\!\left(\tfrac{a}{c};m\right) q^m\right) \\
&= \sum_{m} \left(\frac{1}{\delta}\sum_{r \in R}\sum_{t=0}^{\delta-1}\zeta_\delta^{\delta-tr}Y_S^{(t)}\!\left(\tfrac{a}{c};m\right)\right)q^m = \sum_{m} X_{S,R}\left(\tfrac{a}{c};m\right)q^m.
\end{align*}
It is clear that the non-zero terms come from $m$ in $M$.
\end{proof}
 
 \section{Proof of Theorem \ref{main}} \label{pf2}
 We now reformulate the condition for a partition identity $p_{S_1}(n-H) = p_{S_2}(n)$ to hold for all $n$ in some arithmetic progression in terms of equality of two sieved forms $\mathcal{F}_{S_1',R'}(\tau)$ and $\mathcal{F}_{S_2',R'}(\tau)$. To relate these to the modular functions studied in the previous sections, we need the parameter $H$ to correspond to the difference in the orders of $F_{S_1}(\tau)$ and $F_{S_2}(\tau)$ at infinity,
\begin{equation} \label{defH}
H := H(S_1, S_2) := \text{ord}_{S_1} - \text{ord}_{S_2}.
\end{equation}
Now let $v = \text{Den}(\text{ord}_{S_2})$ and set 
\begin{equation} \label{sprime}
S_1':=vS_1 = \{vs: s \in S_1\} \quad \text{and}  \quad S_2':=vS_2 = \{vs: s \in S_2\}
\end{equation}
with respect to the modulus $\delta' = v\delta$. Then we have 
\[ F_{S_1'}(\tau) = F_{S_1}(v\tau) = q^{v\cdot \text{ord}_{S_2}}\sum_{n=0}^\infty p_{S_1}(n-H)q^{vn},\]
where $p_{S_1}(m) := 0$ for $m < 0$,
and
\[F_{S_2'}(\tau) = F_{S_2}(v\tau) = q^{v \cdot \text{ord}_{S_2}}\sum_{n=0}^\infty p_{S_2}(n)q^{vn} .\]
We would like to obtain an identity of the form
\begin{align} \label{id}
q^{v\cdot \text{ord}_{S_2}} \! \! \! \! \sum_{\substack{n > 0 \\ n \equiv R \! \! \pmod \delta}} \! \! \! \! p_{S_1}(n-H)q^{vn}  =  q^{v\cdot \text{ord}_{S_2}} \! \! \! \! \sum_{\substack{n > 0 \\ n \equiv R\! \! \pmod \delta}}\! \! \! \! p_{S_2}(n)q^{vn}.
\end{align}
We recognize the above as sieved forms with respect to the set
\begin{equation} \label{defR}
R' = \{vr + v\cdot \text{ord}_{S_2} : r \in R\}
\end{equation}
and modulus $\delta'$. More precisely, \eqref{id} is equivalent to the identity
\[\mathcal{F}_{S_1',R'}(\tau) = \mathcal{F}_{S_2',R'}(\tau).\]
By Lemma \ref{2.3}, these forms are modular on
\begin{equation} \label{defG}
\Gamma_{S_1,S_2} :=  \Gamma_1\left(\lcm\left({\delta'}^2, \frac{24\delta'}{\gcd(|S_1|,|S_2|,12)}\right)\right).
\end{equation}
Two modular functions whose poles are supported at the cusps are equal if and only if the principal parts of their expansions agree at every cusp. Since our generating functions have their only poles supported at the cusps, Lemma \ref{4.2} shows that we have $\mathcal{F}_{S_1',R'}(\tau) = \mathcal{F}_{S_2',R'}(\tau)$ if and only if
\[X_{S_1',R'}\left(\tfrac{a}{c};m\right) = X_{S_2',R'}\left(\tfrac{a}{c};m\right)\]
for all $m$ in the finite set
 \begin{equation} \label{M}
 \left\{\mathrm{ord}_{S_i'}^{(t)}\left(\tfrac{a}{c}\right)+\frac{\epsilon^2D^2}{4{\delta'}^2}n < 0:  n \in \mathbb{Z}_{\geq 0}, \  t=0, \ldots, \delta'-1, \ i=1,2 \right\},
 \end{equation}
as $\frac{a}{c}$ runs over all cusp representatives of $\Gamma_{S_1,S_2}$.

 \section{Proof of Theorem \ref{true}} \label{pf1}
 
We prove Theorem \ref{true} using a different method. We first multiply our generating functions by suitable cusp forms to land in a space of holomorphic modular forms. Applying Sturm's bound then reduces the proof to a finite calculation.

 For convenience, set 
\begin{align*}
F_5(\tau) &:= F_{\{1,5,7,9\}}(4 \tau) = q\sum_{n=0}^\infty Q_5(n)q^{4n} 
= \sum_{n=0}^{\infty} a_{5}(n)q^n \\
F_{11}(\tau)&:= F_{\{1,7,9,11\}}(4 \tau) =
q^{-3}\sum_{n=0}^\infty Q_{11}(n)q^{4n}
= \sum_{n=0}^{\infty} a_{11}(n)q^n \\
\intertext{with respect to the modulus $24$, and}
\FF_7(\tau)&:= F_{\{1,7,9,11\}}(5 \tau) 
=q\sum_{n=0}^\infty \mathfrak{Q}_{7}(n)q^{5n}
= \sum_{n=0}^{\infty} \aaa_{7}(n)q^n \\
\FF_{13}(\tau)&:= F_{\{1,9,11,13\}}(5 \tau) 
= q^{-4}\sum_{n=0}^\infty \mathfrak{Q}_{13}(n)q^{5n}
= \sum_{n=0}^{\infty} \aaa_{13}(n)q^n.
\end{align*}
with respect to the modulus $30$.
To prove Theorem 1.1, we must show $a_5(n)=a_{11}(n)$ for all $n \equiv 13 \pmod {24}$ and  $\aaa_{7}(n)= \aaa_{13}(n)$ for $n \equiv 16\pmod {30}$. Our proof relies on a well-known result of Sturm, which states that two holomorphic modular forms of the same weight on the same group are equal if and only if their coefficients agree up to a certain point (see for example Theorem 3.13 of \cite{K}).

\begin{Theorem}[Sturm] \label{sturm}
Let $f(\tau)=\sum_{n=0}^\infty {a_nq^n}$ and $g(\tau)=\sum_{n=0}^\infty {b_nq^n}$ be holomorphic modular forms of weight $k$ on a finite-index subgroup $\Gamma \subseteq \mathrm{SL}_2(\mathbb{Z})$. If $a_n=b_n$ for all 
\[n\leq[\mathrm{SL}_2(\mathbb{Z}):\Gamma] \cdot \frac{k}{12}\]
then  $f(\tau)=g(\tau)$.
\end{Theorem}

Consider the forms
\begin{align*}
\widehat{F}_5(\tau)&:= F_5(\tau)\eta(24\tau)^7 &\qquad \widehat{\FF}_7(\tau)&:= \FF_7(\tau)\eta(30\tau)^{16}\\\widehat{F}_{11}(\tau)&:= F_{11}(\tau)\eta(24\tau)^7 &\qquad \widehat{\FF}_{13}(\tau)&:= \FF_{13}(\tau)\eta(30\tau)^{16}.
\end{align*}
We first show that these forms are holomorphic.
\begin{Lemma}
The forms $\widehat{F}_5(\tau)$ and $\widehat{F}_{11}(\tau)$ are weight $\frac{7}{2}$ holomorphic modular forms on $\Gamma_1(576)$, while $\widehat{\FF}_7(\tau)$ and $\widehat{\FF}_{13}(\tau)$ are weight $8$ holomorphic modular forms on $\Gamma_1(900)$.
\end{Lemma}
\begin{proof}
By Lemma \ref{2.1},  the functions $F_5(\tau)$ and $F_{11}(\tau)$ are modular on $\Gamma_1(576)$, and $\FF_7(\tau)$ and $\FF_{13}(\tau)$ are modular on $\Gamma_1(900)$. The criteria in \eqref{C1} and \eqref{C2} can also be used to show that $\eta(24\tau)$ and $\eta(30\tau)^4$ are modular on $\Gamma_1(576)$ and $\Gamma_1(900)$ respectively. All of the poles and zeros of these functions are supported at cusps.
We can use \eqref{defordt} with $t=0$ to explicitly compute the orders of $F_5(\tau), F_{11}(\tau),$ and $\eta(24\tau)^7$ at all $1152$ cusps of $\Gamma_1(576)$, and do similarly for the orders of $\FF_7(\tau)$, $\FF_{13}(\tau)$, and $\eta(30\tau)^{16}$ at all $2240$ cusps of  $\Gamma_1(900)$. Comparing orders then shows that the products defined above are holomorphic.
\end{proof}

Now define 
\begin{align*}
F^{*}_{5}(\tau) &:= \widehat{F}_5(\tau)|\mathcal{S}_{24,20} &\quad \FF^{*}_{7}(\tau) &:= \widehat{\FF}_7(\tau)|\mathcal{S}_{30,6} \\
F^{*}_{11}(\tau) &:= \widehat{F}_{11}(\tau)|\mathcal{S}_{24,20} &\quad \FF^{*}_{13}(\tau) &:= \widehat{\FF}_{13}(\tau)|\mathcal{S}_{30,6}.
\end{align*}
We claim that the respective pairs of sieved forms are equal.
\begin{Lemma}
We have $F_5^*(\tau) = F_{11}^*(\tau)$ and $\FF_7^*(\tau) = \FF_{13}^*(\tau)$.
\end{Lemma}

\begin{proof}
By Lemma \ref{s}, $F^{*}_{5}(\tau)$ and $F^{*}_{11}(\tau)$ are holomorphic modular forms on $\Gamma_1(576)$, and $\FF^{*}_{7}(\tau)$ and $\FF^{*}_{13}(\tau)$ are holomorphic modular forms on $\Gamma_1(900)$.  Using the formula for the index of $\Gamma_1(N)$ in $\textrm{SL}_2(\mathbb{Z})$ (see Proposition 1.7 of \cite{KO}),
$$[\textrm{SL}_2(\mathbb{Z}):\Gamma_1(N)]=N^2\prod_{\substack{p\mid N\\ p~\text{prime}}}\left(1-\frac{1}{p^2}\right),$$
we find that 
\[[\textrm{SL}_2(\mathbb{Z}):\Gamma_1(576)]=221184 \qquad \text{and} \qquad [\textrm{SL}_2(\mathbb{Z}):\Gamma_1(900)] = 518400.\]
By Theorem \ref{sturm}, it suffices to check the first 64512 coefficients of $F^{*}_{5}(\tau)$ and $F^{*}_{11}(\tau)$, and the first 345600 coefficients of $\FF^{*}_{7}(\tau)$ and  $\FF^{*}_{13}(\tau)$. We have verified that these coefficients agree with a computer.
\end{proof}

To prove the second statement in Theorem \ref{true}, we will also need the following lemma.
\begin{Lemma} \label{L}
A non-constant Laurent polynomial $P(q)$ over $\mathbb{C}$ cannot be a modular function on $\Gamma_1(N)$ for any $N$.
\end{Lemma}

\begin{proof}
Let $P(x)=a_nx^n+\cdots+a_1x+a_0+a_{-1}x^{-1}+\cdots+a_{-k}x^{-k}$ be an arbitrary Laurent polynomial, where each $a_i \in \mathbb{C}$. Assume for the sake of contradiction that $P(q)$ is a modular function on $\Gamma_1(N)$. Then we would have
$$P(q)=P\left(e^{2\pi i(\frac{a\tau+b}{c\tau+d})}\right)$$
for every $\left( \begin{smallmatrix} a&b\\ c&d \end{smallmatrix} \right) \in \Gamma_1(N)$ and every $\tau$ in the upper half-plane. Consider when $\tau=iy$ with $y>0$ real. Taking the limit as $y$ approaches $0$ in the above equation, we find that
$$P(1)=P\left(e^{\frac{2\pi i b}{d}}\right)$$
for every $\left(\begin{smallmatrix} a&b\\ c&d \end{smallmatrix} \right) \in \Gamma_1(N)$. However, the equation $P(x)=P(1)$ has only finitely many solutions, unless $P(x)$ is the constant polynomial $P(1)$. 
\end{proof}

We can now prove Ono's conjecture.
\begin{proof}[Proof of Theorem \ref{true}]
Since $\eta(24\tau)^7$ is supported completely on powers of $q$ which are congruent to $7 $ mod ${24}$, we have
\[F_g^*(\tau) =\left(F_g(\tau)\eta(24\tau)^7\right)|\mathcal{S}_{24,20} = (F_g(\tau)|\mathcal{S}_{24,13})\eta(24\tau)^7 \]
for $g=5,11$ and so
\[F_{5}(\tau) |\mathcal{S}_{24,13} = \frac{F_{5}^*(\tau)}{\eta(24\tau)^7} = \frac{F_{11}^*(\tau)}{\eta(24\tau)^7} = F_{11}(\tau) |\mathcal{S}_{24,13}.\]
This shows that $a_{5}(n) = a_{11}(n)$ for all $n \equiv 13 \pmod {24}$. An identical argument proves that $\mathfrak{a}_{7}(n) = \mathfrak{a}_{13}(n)$ for all $n \equiv 16 \pmod {30}$.

To prove the second claim, let $D(\tau):=F_5(\tau)-F_{11}(\tau)$ and $\DD(\tau):=\FF_7(\tau)-\FF_{13}(\tau)$. Suppose for the sake of contradiction there exist only finitely many $n \equiv 0,2 \pmod 6$ such that  $Q_5(n-1) \neq Q_{11}(n)$ and finitely many $m \equiv 0, 2 \pmod 6$ such that $\QQ_7(m-1) \neq  \QQ_{13}(m)$. Then after sieving on the arithmetic progressions $n \equiv 5,21 \pmod {24}$ and $m \equiv 6,26 \pmod {30}$, the resulting functions 
\[D(\tau)|\mathcal{S}_{24,5}+D(\tau)|\mathcal{S}_{24,21} \quad \text{and} \quad \DD(\tau)|\mathcal{S}_{30,6}+\DD(\tau)|\mathcal{S}_{30,26}\]
are non-constant Laurent polynomials in $q$. By Lemma \ref{s}, the above functions are modular on $\Gamma_1(576)$ and $\Gamma_1(900)$ respectively, which is impossible by Lemma \ref{L}.
\end{proof}

\begin{remark}

An anonymous referee has informed the authors of an alternative proof of Theorem 1.1, which we sketch here. Let $G_5(\tau) = q^{-1/4}F_{\{1,5,7,9\}}(\tau+\frac{1}{2})$, and let $G_{11}(\tau)=q^{3/4}F_{\{1,7,9,11\}}(\tau+\frac{1}{2})$.
We claim that
\begin{equation}
\label{altproof}G_{11}(q)+qG_5(q)= \prod_{n \equiv \pm 2, \pm 6, \pm 8, \pm10 \pmod{24}}(1-q^n).
\end{equation}
To prove this identity, we divide both sides by $G_{11}(q)$ and check that both sides are modular functions on $\Gamma_1(24)$ via \eqref{C1} and \eqref{C2}. Since the number of zeros of a modular function must equal the number of poles (counted with multiplicity), if one finds explicit bounds for the orders of poles at all cusps and shows that the difference of these functions vanishes to high enough order at infinity, the two functions must be equal. Carrying out this procedure, one needs only to check coefficients of the expansion at infinity through $q^5$. Using elementary methods, one can also see that the coefficient of $q^{6n+4}$ in the right-hand side of \eqref{altproof} always vanishes, which implies that $Q_5(n-1)=Q_{11}(n)$ for all $ n \equiv 4\pmod{6}$. 

A similar process can be used to prove $\QQ_{7}(n-1)=\QQ_{13}(n)$ for $n \equiv 4 \pmod 6$.

\end{remark}

\end{document}